\documentclass[a4paper,11pt,twoside]{amsart}
\usepackage[a4paper]{geometry}
\usepackage{amsmath,amsthm}
\usepackage{amssymb,amsfonts}
\usepackage{hyperref}
\usepackage{indentfirst}
\usepackage{tikz}
\usepackage{epigraph}
\usepackage{enumerate}
\usepackage{graphicx}
\usepackage{calligra}
\usepackage{stmaryrd}
\usepackage{chngpage}
\DeclareSymbolFont{bbold}{U}{bbold}{m}{n}
\DeclareSymbolFontAlphabet{\mathbbm}{bbold}

\title{Voiculescu's Theorem for Nonseparable \cstar-algebras}

\theoremstyle{plain}
	\newtheorem{theorem*}{Theorem}
	\newtheorem*{VT*}{Voiculescu's Theorem}
	\newtheorem*{GL*}{Glimm's Lemma}
	\newtheorem{theorem}{Theorem}[section]
	\newtheorem{proposition}[theorem]{Proposition}
	\newtheorem{lemma}[theorem]{Lemma}
	
	\newtheorem{corollary}[theorem]{Corollary}
	\newtheorem{corollary*}[theorem*]{Corollary}

\theoremstyle{definition}
	\newtheorem*{NP*}{Naimark's problem}
	\newtheorem*{dia*}{The diamond principle ($\diamondsuit$)}
	\newtheorem*{MA*}{$\textsf{MA}_\kappa$($\sigma$-centered)}

	\newtheorem*{acknow}{Acknowledgements}
\theoremstyle{remark}

\DeclareMathOperator{\Fin}{Fin}

\newcommand{\C}{\mathbb{C}}

\newcommand{\N}{\mathbb{N}}

\newcounter{my_enumerate_counter}

\newcommand{\set}[1]{\{#1\}}					

\newcommand{\cstar}{$\mathrm{C}^\ast$}
\newcommand{\A}{\mathcal{A}}

\usepackage[mode=multiuser,status=draft,lang=english]{fixme}
\fxsetup{theme=colorsig}
\hypersetup{colorlinks=false, pdfborder={0 0 0}}

\author{Andrea Vaccaro}
\address[A. Vaccaro]{Department of Mathematics, University of Pisa, Largo Bruno Pontecorvo 5
    Pisa, Italy, 56127 - Department of Mathematics and Statistics, York University, 4700 Keelee Street,
    North York, Ontario, Canada, M3J 1P3}
\email[]{vaccaro@mail.dm.unipi.it}
\urladdr{http://people.dm.unipi.it/vaccaro/index.html}

\date{}
\keywords{Voiculescu's theorem, nonseparable \cstar-algebras, Martin's axiom.}

\begin{document}
\maketitle

\begin{abstract}
We prove that Voiculescu's noncommutative
version of the Weyl-von Neumann theorem can be extended to all
(not necessarily separable) unital, separably representable
 \cstar-algebras whose density character is strictly smaller
than $\mathfrak{p}$.
We show moreover that Voiculescu's theorem
consistently does not generalize to \cstar-algebras of larger density character.
\end{abstract}

\section{Introduction} \label{s0}
Voiculescu's work in \cite{voiculescu}, a far-reaching extension of the results
by Weyl and von Neumann on unitary equivalence up to compact perturbation of
self-adjoint operators (\cite{weyl}, \cite{vNeu}), is one of the cornerstones of
the theory of extensions of separable \cstar-algebras (see \cite{david} for
a survey of these results).

The label `Voiculescu's theorem' often refers to a
collection of results and corollaries from
\cite{voiculescu}, rather than a specific theorem.
Through this paper, it always refers to the following statement, where
for a complex Hilbert space $H$,
$\mathcal{B}(H)$ is the algebra of linear bounded operators from
$H$ into itself, and $\mathcal{K}(H)$ is the algebra of compact operators.
\begin{VT*} \label{voict}
Let $H, L$ be two separable Hilbert spaces, $\A \subseteq \mathcal{B}(H)$ a
separable unital \cstar-algebra and $\sigma : \A \to \mathcal{B}(L)$ a
unital completely positive map such that
$\sigma(a) = 0$ for all $a \in \A \cap \mathcal{K}(H)$. Then there is a sequence of isometries
$V_n: L \to H$ such that $\sigma(a) - V_n^* a V_n \in \mathcal{K}(L)$ and
$\lim_{n \to \infty} \lVert \sigma(a) - V_n^* a V_n \rVert = 0$ for all $a \in \A$.
\end{VT*}
See \S \ref{s1} for a definition of completely positive map. We remark that all $*$-homomorphisms, as well as the states
of any given \cstar-algebra, are completely positive (see \cite[Example 1.5.2]{brownozawa}). The specific instance of Voiculescu's theorem when $L =\C$ and
$\sigma$ is a state is known as Glimm's lemma (\cite[Lemma 3.6.1]{higsonroe}),
whose statement is the following.
\begin{GL*}
Let $H$ be a separable Hilbert space, $\A \subseteq \mathcal{B}(H)$ a separable
unital \cstar-algebra and $\sigma: \A \to \C$ a state such that $\sigma(a) = 0$
for all $a \in \A \cap \mathcal{K}(H)$. There exists a sequence of orthonormal
vectors $\set{\xi_n}_{n \in \N}$ such that $\sigma(a) = \lim_{n \to \infty} \langle a \xi_n,
\xi_n \rangle$ for every $a \in \A$.
\end{GL*}
We refer the reader to \cite{arveson} and \cite[\S 3.4--3.6]{higsonroe} for a proof
of these classical results.

The study of the extensions of a \cstar-algebra and of the invariant
Ext, one of the main fields of applications of Voiculescu's theorem,
flourished after the seminal work in \cite{bdf}
(see \cite{higsonroe} or \cite{blackk} for an introduction to this subject).
Given a unital \cstar-algebra $\A$, $\text{Ext}(\A)$ is the set of
all unital embeddings of $\A$ into the Calkin algebra $\mathcal{Q}(H)$
(modulo the relation of unitary equivalence). The set $\text{Ext}(\A)$
can be endowed with a semigroup structure, and one of the main consequences of
Voiculescu's theorem in this framework is
that $\text{Ext}(\A)$ has an identity when $\A$ is separable and unital.
This, along with the results in \cite{choieffros}, implies for instance that
$\text{Ext}(\A)$ is a group for every nuclear separable unital \cstar-algebra $\A$.

Voiculescu's theorem has also been recently employed in combination with
set theory in the study of which nonseparable \cstar-algebras embed into the
Calkin algebra, in \cite{farahvignati} and \cite{fkv}.
The current paper pushes further the interaction of the results in \cite{voiculescu}
with set theory, consistently extending Voiculescu's theorem to certain
`small' nonseparable \cstar-algebras.

Let $H, L$ be two separable Hilbert spaces,
$\A \subseteq \mathcal{B}(H)$ a nonseparable unital \cstar-algebra and $\sigma : \A \to
\mathcal{B}(L)$ a unital completely positive map such that $\sigma(a) = 0$ for
all $a \in \A \cap \mathcal{K}(H)$. If we do not assume anything else,
all that Voiculescu's theorem can guarantee is the existence of a net of isometries
$\set{V_\lambda}_{\lambda \in \Lambda}$ from $L$ into $H$
such that, for any separable subalgebra
$\mathcal{B} \subset \A$ and $\epsilon > 0$, there is $\mu \in \Lambda$ such that 
$\lambda \ge \mu$ implies $\sigma(a) - V_\lambda^* a V_\lambda \in \mathcal{K}(L)$ and
$\lVert \sigma(a) - V_\lambda^* a V_\lambda \rVert < \epsilon$ for all $a \in \mathcal{B}$.
In this note we prove the following theorem.
\begin{theorem*} \label{t1}
Let $H$ be a separable Hilbert space.
\begin{enumerate}
\item \label{i1t1} Let $L$ be a separable Hilbert space,
$\A \subseteq \mathcal{B}(H)$ a unital \cstar-algebra of density character strictly
less than $\mathfrak{p}$\footnote{The cardinal invariant $\mathfrak{p}$ is
the least size of a centered subfamily $F$ of $\mathcal{P}(\N) / \Fin$
which does not have a lower bound in $\mathcal{P}(\N) / \Fin$ (for an introduction
to this cardinal invariant and its basic properties see \cite{bart}).}
and $\sigma : \A \to \mathcal{B}(L)$ a unital completely positive map such that
$\sigma(a) = 0$ for all $a \in \A \cap \mathcal{K}(H)$.
Then there is a sequence of isometries
$V_n: L \to H$ such that $\sigma(a) - V_n^* a V_n$ is compact and
$\lim_{n \to \infty} \lVert \sigma(a) - V_n^* a V_n \rVert = 0$ for all $a \in \A$.
\item \label{i2t1} Given a cardinal $\lambda$, it is consistent with $\textsf{ZFC} + \mathfrak{c} \ge \lambda$ (where $\mathfrak{c}$ is the cardinality of the continuum)
that there exist a unital
\cstar-algebra $\A \subseteq \mathcal{B}(H)$ of density character less
than $\mathfrak{c}$, and $\sigma$, a state of $\A$ annihilating
$\A \cap \mathcal{K}(H)$, for which Glimm's lemma fails.
\end{enumerate}
\end{theorem*}

Theorem \ref{t1} gives the following corollary.
\begin{corollary*} \label{c}
The statement `Voiculescu's theorem holds for all separably representable
\cstar-algebras of density character less than $\mathfrak{c}$' is independent from
\textsf{ZFC}. Moreover, it is independent from $\textsf{ZFC} + \mathfrak{c} \ge \lambda$
for any cardinal $\lambda$.
\end{corollary*}
\begin{proof}
Under Martin's axiom, which implies $\mathfrak{p} = \mathfrak{c}$,
the statement holds by item \ref{i1t1} of theorem \ref{t1}.
The consistent failure of the statement follows by item \ref{i2t1} of theorem \ref{t1}.
\end{proof}

The argument used to obtain the first part of theorem \ref{t1} is inspired to the proof of
Voiculescu's theorem as given by Arveson in \cite{arveson}
(see also \cite[\S 3.4--3.6]{higsonroe}). We show that such proof
essentially consists of a sequence of diagonalization arguments
which are equivalent to applications of the Baire category theorem to certain
$\sigma$-centered partial orders (see \S \ref{s1} for a definition).
Item \ref{i1t1} of theorem \ref{t1}
then follows by the results in \cite{bell}, where it is shown that
Martin's axiom holds for $\kappa$-sized families of dense subsets of $\sigma$-centered
partial orders if and only if $\kappa < \mathfrak{p}$.

The second item of theorem \ref{t1} is obtained via an application of Cohen's forcing
and a simple cardinality argument. Starting from a \cstar-algebra $\A$ of density character
$\mathfrak{c}$ for which Glimm's lemma fails, we show that the lemma still
fails for $\A$ also after adding enough (but not too many) Cohen reals.

We remark that Voiculescu's theorem is false in general for
subalgebras of $\mathcal{B}(H)$ of denisty character $\mathfrak{c}$,
as witnessed by $\ell^\infty(\N)$, $L^\infty([0,1])$ and $\mathcal{B}(H)$ itself (see \S \ref{s3}). We do not know if the notion of smallness given by $\mathfrak{p}$ in this context
is optimal, or if it is consistent that there are \cstar-algebras of density character greater
than or equal to $\mathfrak{p}$ for which the conclusion of Voiculescu's theorem holds.

The paper is organized as follows. Section \ref{s1} is devoted to definitions
and preliminaries. In section \ref{s2} we prove item \ref{i1t1} of theorem \ref{t1} and list some standard corollaries
of Voiculescu's theorem which generalize to \cstar-algebras of density character
smaller than $\mathfrak{p}$.
Finally in \S \ref{s3} we give a proof of item \ref{i2t1}.

\section{Preliminaries} \label{s1}
Through this paper,
given a complex Hilbert space $H$,
$\mathcal{B}(H)$ is the algebra of linear bounded operators from
$H$ into itself, and $\mathcal{K}(H)$ is the algebra of compact operators.
The Calkin algebra $\mathcal{Q}(H)$ is the quotient $\mathcal{B}(H) / \mathcal{K}(H)$.

An \emph{approximate unit} of a \cstar-algebra $\A$ is a net
$\set{h_\lambda}_{\lambda \in \Lambda}$
of positive contractions of $\A$ such that $\lim_{\lambda} \lVert h_\lambda a - a
\rVert = \lim_{\lambda} \lVert a h_\lambda - a \rVert = 0$ for all $a \in \A$.
Given a \cstar-algebra $\A \subseteq \mathcal{B}(H)$, an approximate unit of $\mathcal{K}(H)$
is \emph{quasicentral} for $\A$ if
$\lim_\lambda \lVert [h_\lambda, a ] \rVert = 0$ for all $a \in \A$, where
$[b,c]$ denotes, for two operators $b,c \in \mathcal{B}(H)$, the commutant $bc- cb$.

For a \cstar-algebra $\A$, let $M_n(\A)$ be the \cstar-algebra of $n \times n$ matrices
with entries in $\A$.
Given two \cstar-algebras $\A$ and $\mathcal{B}$, a bounded linear
map $\sigma: \A \to \mathcal{B}$ is \emph{completely positive} if for all $n \in \N$
all the maps $\phi_n : M_n (\A) \to M_n(\mathcal{B})$, defined as
\[
\phi_n([a_{ij}]) = [\phi(a_{ij})]
\]
are positive, i.e. they send positive elements into positive elements.

The notation $F \Subset G$ stands for `$F$ is a finite subset of $G$'.

A \emph{partially ordered set} (or simply \emph{poset}) $(\mathbb{P}, \le)$
is a set equipped with a binary, transitive, antisymmetric, reflexive relation $\le$.
The poset $(\mathbb{P}, \le)$ is
\emph{centered} if for any $F \Subset \mathbb{P}$ there is
$q \in \mathbb{P}$ such that $q \le p$ for all $p \in F$, and it is
$\sigma$\emph{-centered} if it is the union of countably many centered sets.

A subset $D \subseteq \mathbb{P}$ is \emph{dense} if for every $p \in \mathbb{P}$
there is $q \in D$ such that $q \le p$.
A set $G \subseteq \mathbb{P}$ is a \emph{filter}
if $q \in G$ and $q \le p$ implies $p \in G$, and if for any $p, q \in G$ there is $r \in G$
such that $r \le p$, $r \le q$. Given $\mathcal{D}$ a collection of dense subsets of $\mathbb{P}$,
a filter $G$ is $\mathcal{D}$\emph{-generic} if $G \cap D \not= \emptyset$ for all $D \in
\mathcal{D}$. By the results in \cite{bell}, $\kappa < \mathfrak{p}$ is equivalent
to the following weak form of Martin's axiom.

\begin{MA*}
Given a $\sigma$-centered poset $(\mathbb{P} , \le)$
and $\mathcal{D}$ a collection of size $\kappa$ of dense subsets of $\mathbb{P}$,
there exists a $\mathcal{D}$-generic filter on $\mathbb{P}$.
\end{MA*}

Before moving to the proof of theorem \ref{t1}, we prove a simple preliminary fact.
It is known that for every \cstar-algebra $\A \subseteq \mathcal{B}(H)$
there is an approximate unit
of the compact operators which is quasicentral for $\A$ (see \cite[Theorem 1 p.330]{arveson}).
Moreover, if $\A$ is separable,
the quasicentral approximate unit can be chosen to be countable, hence sequential. This property can be generalized to all \cstar-algebras of density character $\kappa$ 
such that $\textsf{MA}_\kappa$($\sigma$-centered) holds. This is a simple fact, nevertheless its proof gives a fairly clear idea, at least to
the reader familiar with the proof of Voiculescu's theorem given in \cite{arveson},
of how to prove item \ref{i1t1} of theorem \ref{t1}.

\begin{proposition} \label{qc}
Let $H$ be a separable Hilbert space and $\A \subseteq \mathcal{B}(H)$
a \cstar-algebra of density character less than $\mathfrak{p}$. Then there
exists a sequential approximate unit $\set{h_n}_{n \in \N}$ of $\mathcal{K}(H)$ which
is quasicentral for $\A$.
\end{proposition}
\begin{proof}
Fix a countable dense $K$ in the set of all positive norm one elements of $\mathcal{K}(H)$,
and $B$ dense in $\A$ of size smaller than $\mathfrak{p}$.
Let $\mathbb{P}$ be the set of tuples
\[
p = (F_p, J_p, n_p, (h^p_j)_{j \le n_p})
\]
where $F_p \Subset B$, $J_p \Subset K$, $n_p \in \N$
and $h^p_j \in K$ for all $j \le n_p$.
For $p, q \in \mathbb{P}$ we say $p \le q$ if and only if
\begin{enumerate}
\item $F_q \subseteq F_p$,
\item $J_q \subseteq J_p$,
\item $n_q \le n_p$,
\item  $h^p_j = h^q_j$ for all $j \le n_q$,
\item \label{itemqc} if $n_q < n_p$ then, for all $n_q < j \le n_p$, all $k \in J_q$ and
all $a \in F_q$, the following holds
\[
\max \set{ \lVert [ h_j^p , a] \rVert ,  \lVert h^p_j k  - k \rVert, \lVert kh^p_j   - k \rVert} < 1/j
\]
\end{enumerate}
The poset $(\mathbb{P}, \le)$ is $\sigma$-centred since, for
any finite $X\Subset \mathbb{P}$ such that
there is $n \in \N$ and $(h_j)_{j \le n} \in K^n$ satisfying $n_p = n$ and
$(h^p_j)_{j \le n} = (h_j)_{j \le n}$ for all $p \in X$, the condition
\[
r = \left( \bigcup_{p \in X} F_{p}, \bigcup_{p \in X} J_{p}, n, (h_j)_{j \le n}\right)
\]
is a lower bound for $X$. Let $\mathcal{D}$ be the collection
of the sets
\[
\Delta_{F, J,n} = \set{p \in \mathbb{P} : F_p \supseteq F, J_p \supseteq J,  n_p \ge n}
\]
for $F \Subset B$, $J \Subset K$ and $n \in \N$. The sets $\Delta_{F, J,n}$ are
dense since for every separable subalgebra of $\mathcal{B}(H)$
there is a sequential approximate unit of $\mathcal{K}(H)$ which is quasicentral for it.
A $\mathcal{D}$-generic filter produces a sequential approximate
unit of $\mathcal{K}(H)$ which is quasicentral for $\A$. Such filter exists by
$\textsf{MA}_{\lvert \mathcal{D} \rvert}$($\sigma$-centered), which holds since
$\mathcal{D}$ has size smaller than $\mathfrak{p}$. 
\end{proof}

\section{Voiculescu's Theorem and Martin's Axiom} \label{s2}
Similarly to what happens in \cite{arveson}, we split
the proof of item \ref{i1t1} of theorem \ref{t1} in two steps. First we prove the statement
assuming that the completely positive map $\sigma$ is block-diagonal
(see lemma \ref{voiclemma1}),
then in lemma \ref{voiclemma2} we show that the general case can be reduced to the
block-diagonal case.

\subsection{Block-Diagonal Maps}
A completely positive map $\sigma: \A \to \mathcal{B}(L)$ is \emph{block-diagonal}
if there is a decomposition $L = \bigoplus_{n \in \N} L_n$, where $L_n$ is
finite-dimensional for all $n \in \N$, which in turn induces a decomposition
$\sigma = \bigoplus_{n \in \N} \sigma_n$ where the maps
$\sigma_n : \A \to \mathcal{B}(L_n)$ are completely positive.

\begin{lemma} \label{voiclemma1}
Let $H, L$ be two separable Hilbert spaces,
$\A \subseteq \mathcal{B}(H)$ a unital \cstar-algebra
of density character less than $\mathfrak{p}$ and $\sigma: \A \to B(L)$
a block-diagonal, unital, completely positive map such that $\sigma(a) = 0$ for all $a \in \A \cap
\mathcal{K}(H)$.
Then there is a sequence of isometries
$V_n: L \to H$ such that $\sigma(a) - V_n^* a V_n \in \mathcal{K}(L)$ and
$\lim_{n \to \infty} \lVert \sigma(a) - V_n^* a V_n \rVert = 0$ for all $a \in \A$.
\end{lemma}
\begin{proof}
Fix $\epsilon > 0$.
By hypothesis $L = \bigoplus_{n \in \N} L_n$, with $L_n$
finite-dimensional for all $n \in \N$, and $\sigma$ decomposes as
$\bigoplus_{n \in \N} \sigma_n$, where $\sigma_n(a) = 0$ whenever $a \in \A \cap
\mathcal{K}(H)$ for all $n \in \N$.
Let $K$ be a countable dense subset of the unit ball of $H$ such that, for
every $\xi \in K$ the set $\set{ \eta \in K : \eta \perp \xi}$ is dense in $\set{
\eta \in H : \lVert \eta \rVert = 1, \ \eta \perp \xi}$.
Fix an orthonormal basis $\set{\xi^n_j}_{j \le k_n}$ for each $L_n$.
Consider the set $\mathbb{P}$ of the tuples
\[
p = (F_p, n_p, (W^p_i)_{i \le n_p})
\]
where $F_p \Subset \A$, $n_p \in \N$ and $W^p_i$ is an isometry of $L_i$ into $H$
such that $W^p_i \xi^i_j \in K$ for every $j \le k_i$ and $i \le n_p$.
We say $p \le q$ for two elements in $\mathbb{P}$ if and only if
\begin{enumerate}
\item $F_q \subseteq F_p$,
\item $n_q \le n_p$,
\item $W^p_i = W^q_i$ for all $i \le n_q$,
\item \label{i4bd} for $ n_q < i \le n_p$ (if any) we require $W_i L_i$ to be orthogonal to
$\set{ W_j L_j,  a W_j L_j,  a^* W_j L_j : j \le i, \ a \in F_q}$ and
\[
\lVert \sigma_i(a) - W^*_i a W_i \rVert < \epsilon/2^{i}
\]
for all $a \in F_q$.
\end{enumerate}
For any finite set of conditions $X \Subset \mathbb{P}$ such that
there is $n \in \N$ and $(W_i)_{i \le n}$ satisfying
$n_p = n$ and $(W^p)_{i \le n_p} = (W_i)_{i \le n}$ for all $p \in X$, the condition
\[
r = \left( \bigcup_{p \in X} F_p, n, (W_i)_{i \le n} \right)
\]
is a lower bound for $X$. Thus the poset $(\mathbb{P}, \le)$
is $\sigma$-centered. Let $\mathcal{D}$ be the collection of the sets
\[
\Delta_{F, n} = \set{p \in \mathbb{P} : F_p \supseteq F, n_p \ge n}
\]
for $n \in \N$ and $F \Subset B$, where $B$ is a fixed
dense subset of $\A$ of size smaller than $\mathfrak{p}$. By theorem \ref{voict} every
$\Delta_{F,n}$ is dense in $\mathbb{P}$ (the orthogonality
condition in item \ref{i4bd} of the definition of the order relation can be
obtained using proposition 3.6.7 in \cite{higsonroe}).
Let $G$ be a $\mathcal{D}$-generic filter, which exists since $\lvert \mathcal{D} \rvert
< \mathfrak{p}$, and thus $\textsf{MA}_{\lvert \mathcal{D} \rvert}$($\sigma$-centered),
holds. Let $V$ be the isometry from $\bigoplus L_n$ into $H$ defined
as $\bigoplus_{n \in \N} W_n$ where
$W_n = W_n^p$ for some $p \in G$ such that $n_p \ge n$. The isometry is well defined
since $G$ is a filter. The proof that $\sigma(a) - V^* a V \in \mathcal{K}(L)$
and that $\lVert \sigma(a) - V^* a V \rVert < \epsilon$ for all $a \in \A$ is the same
as in lemma 3.5.2 in \cite{higsonroe}.
\end{proof}

\subsection{The General Case}
The following lemma generalizes theorem 3.5.5 of \cite{higsonroe} to all
\cstar-algebras of density character smaller than $\mathfrak{p}$.

\begin{lemma} \label{voiclemma2}
Let $H, L, L'$ be separable Hilbert spaces,
$\A \subseteq \mathcal{B}(H)$ a unital \cstar-algebra of density character less
than $\mathfrak{p}$ and $\sigma: \A \to B(L)$ a unital completely positive map
such that $\sigma(a) = 0$ for all $a \in \A \cap \mathcal{K}(H)$. Then there
is a block-diagonal, unital completely positive map $\sigma': \A \to B(L')$,
such that $\sigma'(a) = 0$ for all $a \in \A \cap \mathcal{K}(H)$, and a sequence of isometries $V_n : H \to L$ such that
$\sigma(a) - V_n^* \sigma'(a) V_n \in \mathcal{K}(H)$ and $\lim_{n \to \infty} \lVert
\sigma(a) - V_n^* \sigma'(a) V_n \rVert = 0$ for all $a \in \A$.
\end{lemma}
\begin{proof}
Fix $\epsilon > 0$. We use the same poset (and notation) defined in proposition \ref{qc}
to generate an approximate unit of $\mathcal{K}(H)$ which is quasicentral
for $\sigma[\A]$. Adjusting suitably
the inequality in item \ref{itemqc} of the definition of such poset
(see \cite[Lemma p.332]{arveson}), by $\textsf{MA}_{\lvert \mathcal{D} \rvert}$($\sigma$-centered)
there is a filter of $\mathbb{P}$ which generates an approximate unit
$(h_n)_{n \in \N}$ such that if
$a \in F_p$ for some $p \in G$, then for all $n > n_p$ we have
\[
\lVert [(h_{n+1} - h_n)^{1/2}, \sigma(a)]  \rVert < \epsilon/2^n
\]
From this point the proof goes verbatim as in theorem 3.5.5 of \cite{higsonroe}.
\end{proof}

The thesis of item \ref{i1t1} of theorem \ref{t1} follows composing the isometries obtained
from lemmas \ref{voiclemma1} and \ref{voiclemma2}.

\subsection{Corollaries and Remarks}
Voiculescu's theorem allows to infer several corollaries if the completely
positive map $\sigma$ is assumed to be a $*$-homomorphism.
Using item \ref{i1t1} of theorem \ref{t1}, these results generalize to
separably representable \cstar-algebras of density character less than $\mathfrak{p}$.
We omit the proofs in this part as they can be obtained following
verbatim the arguments used in the separable case.

We introduce some definitions to ease the notation in the following statements.
Given a representation $\phi : \A \to \mathcal{B}(H)$, let $H_e$ be the Hilbert
space spanned by $(\phi[\A] \cap \mathcal{K}(H))H$. Since
$\phi[\A] \cap \mathcal{K}(H)$ is an ideal
of $\phi[A]$, the space $H_e$ is invariant for $\phi[A]$. The \emph{essential part of}
$\phi$, denoted $\phi_e$, is the restriction of $\phi$ to $H_e$.

Two representations $\phi: \A \to \mathcal{B}(H_1)$ and $\psi : \A \to \mathcal{B}(H_2)$
are \emph{equivalent} if there is a unitary map $U: H_1 \to H_2$ such
that $U^* \psi(a) U = \phi(a)$ for all $a \in \A$. They
are \emph{approximately equivalent} if there is a sequence of unitary maps
$U_n : H_1 \to H_2$ such that $U_n^* \psi(a) U_n - \phi(a) \in
\mathcal{K}(H_1)$ and $\lim_{n \to \infty} \lVert U_n^* \psi(a) U_n - \phi(a) \rVert = 0$
for all $a \in \A$. Finally, they
are \emph{weakly approximately equivalent} if there are two sequences of
unitary maps $U_n : H_1 \to H_2$ and $V_n : H_2 \to H_1$
such that $U_n^* \psi(a) U_n \to \phi(a)$ and $V_n^* \phi(a) V_n \to \psi(a)$
in the weak operator topology.

Corollaries \ref{voicMA2} and \ref{voicMA3} can be proved
using the proofs of \cite[Corollary 2 p. 339]{arveson} and
\cite[Theorem 5]{arveson} plus \cite[Corollary 1 p. 343]{arveson} respectively,
after substituting all the instances of Voiculescu's theorem with item \ref{i1t1} of theorem
\ref{t1}.

\begin{corollary} \label{voicMA2}
Let $H, L$ be two separable Hilbert spaces,
$\A \subseteq \mathcal{B}(H)$ a unital \cstar-algebra
of density character less than $\mathfrak{p}$ and $\phi: \A \to B(L)$
a unital representation such that $\phi(a) = 0$ for all $a \in \A \cap \mathcal{K}(H)$.
Then the direct sum representation $\text{Id} \oplus \phi$ on $H \oplus L$ is approximately
equivalent to $\phi$.
\end{corollary}

\begin{corollary} \label{voicMA3}
Let $\A$ be a separably representable unital \cstar-algebra of density less than
$\mathfrak{p}$
and $\phi, \psi$ two unital representations on some separable, infinite dimensional
Hilbert space $H$. The following are equivalent.
\begin{enumerate}
\item $\phi$ and $\psi$ are approximately equivalent,
\item $\phi$ and $\psi$ are weakly approximately equivalent,
\item $\text{ker}(\phi) = \text{ker} (\psi)$, $\text{ker} (\pi \circ \phi) =
\text{ker} (\pi \circ \psi)$ (here
$\pi: \mathcal{B} (H) \to \mathcal{Q}(H)$ is the quotient map)
and $\phi_e$ is equivalent to $\psi_e$.
\end{enumerate}
In particular, if $\text{ker}(\phi) = \text{ker} (\psi)$ and $\phi[A] \cap \mathcal{K}(H) =
\psi[A] \cap \mathcal{K}(H) = \set{0}$ then $\phi$ and $\psi$ are approximately equivalent.
\end{corollary}

A further consequence of Voiculescu's theorem is that every separable unital subalgebra
of the Calkin algebra is equal to its double commutant in the Calkin algebra (see
\cite[p. 345]{arveson}; see also \cite{bic} for a version of this statement in the context
of ultrapowers). It is not clear whether $\textsf{MA}_\kappa$($\sigma$-centered) could be used to generalize
this fact to \cstar-algebras of density character $\kappa$,
even assuming they are separably representable.

\section{Independence} \label{s3}
In this section we prove item \ref{i2t1} of theorem \ref{t1}.
\begin{proof}[Proof of item \ref{i2t1} of theorem \ref{t1}]
Let $H$ be a separable Hilbert space and
$\A \subseteq \mathcal{B}(H)$
a maximal abelian atomic subalgebra, hence isomorphic to $\ell^\infty(\N)$.
Since the pure states of $\A$ annihilating $\A \cap \mathcal{K}(H)$ are
in bijection with the non-principal ultrafilters on $\N$ (see \cite[Example 6.2]{setoperator}),
there are $2^\mathfrak{c}$ of them. Since there are only $\mathfrak{c}$ (countable)
sequences of vectors in $H$,
there are $2^\mathfrak{c}$ states of $\A$ for which Glimm's lemma
fails (it actually fails for all pure states annihilating
$\A \cap \mathcal{K}(H)$, as shown in proposition 2.7 of \cite{hadwin}). We prove
the statement of item \ref{i2t1} of theorem \ref{t1} for $\lambda = \aleph_2$,
as the proof in the general case is analogous.
Consider a model of \textsf{ZFC} where $\mathfrak{c} = \aleph_1$ and
$2^{\aleph_1} = \aleph_3$ and add to it $\aleph_2$ Cohen reals.
In the generic extension we have $\mathfrak{c} = \aleph_2$, thus (the closure of)
$\A$ has density character strictly smaller than $\mathfrak{c}$. Glimm's lemma fails
for $\A$ also in the generic extension.
There are in fact at most $ \aleph_2$ new sequences of vectors of $H$,
which are still not enough to cover all the $\aleph_3$ states of $\A$ for which
Glimm's lemma failed in the ground model.
\end{proof}
The argument we just exposed can be generalized verbatim to other \cstar-algebras of
density character $\mathfrak{c}$ such as $\mathcal{B}(H)$ or $L^\infty([0,1])$, which all
have more than $2^{\mathfrak{c}}$ different states.\footnote{Notice that, if $V$ is the ground
model of \textsf{ZFC} and $V[G]$ a generic extension, the closure of $\mathcal{B}(H)^V$
in $V[G]$ is generally strictly contained in $\mathcal{B}(H)^{V[G]}$. The same happens for
$\ell^\infty(\N)$ and $L^\infty([0,1])$.}

\begin{acknow}
I would like to thank Ilijas Farah for the interesting conversations on this topic we had and for
his useful suggestions on the earlier drafts of this paper.
\end{acknow}

\bibliographystyle{amsalpha}
	\bibliography{Bibliography}

\end{document}